\documentclass[10pt]{article}
 
\usepackage[margin=1in]{geometry} 
\usepackage{amsmath,amsthm,amssymb}
\usepackage{graphicx}
\usepackage{wrapfig}
\usepackage{wasysym}
\usepackage[utf8]{inputenc}
\usepackage{tikz}
\usetikzlibrary{arrows}
\usepackage{caption}
\usepackage{titlesec}
\usepackage{mathrsfs}
\usepackage{enumitem}
\usepackage{cleveref}
\usepackage{float}
\usepackage[affil-it]{authblk}
\usepackage{commath}

\crefname{section}{§}{§§}
\Crefname{section}{§}{§§}

\titleformat{\section}[block]{\large\bfseries\filcenter}{}{1em}{}
\titleformat{\subsection}[hang]{\bfseries}{}{1em}{}
\newcommand{\cercle}[4]{
\node[circle,inner sep=0,minimum size={2*#2}](a) at (#1) {};
\draw[blue,thick] (a.#3) arc (#3:{#3+#4}:#2);
}

\newenvironment{theorem}[2][Theorem]{\begin{trivlist}
\item[\hskip \labelsep {\bfseries #1}\hskip \labelsep {\bfseries #2.}]}{\end{trivlist}}
\newenvironment{lemma}[2][Lemma]{\begin{trivlist}
\item[\hskip \labelsep {\bfseries #1}\hskip \labelsep {\bfseries #2.}]}{\end{trivlist}}

\newenvironment{proposition}[2][Proposition]{\begin{trivlist}
\item[\hskip \labelsep {\bfseries #1}\hskip \labelsep {\bfseries #2.}]}{\end{trivlist}}

\newenvironment{definition}[2][Definition]{\begin{trivlist}
\item[\hskip \labelsep {\bfseries #1}\hskip \labelsep {\bfseries #2.}]}{\end{trivlist}}
\newenvironment{example}[2][Example]{\begin{trivlist}
\item[\hskip \labelsep {\bfseries #1}\hskip \labelsep {\bfseries #2.}]}{\end{trivlist}}

\usepackage{blindtext}
\usepackage[utf8]{inputenc}
 
\title{The largest eigenvalue of a convex function, duality, and a theorem of Slodkowski}
\author{Matthew Dellatorre}
\date{ }
 
\begin{document}
\maketitle
\begin{abstract}
First, we provide an exposition of a theorem due to Slodkowski regarding the largest \textquotedblleft eigenvalue\textquotedblright \hspace{.5mm} of a convex function. In his work on the Dirichlet problem, Slodkowski introduces a generalized second-order derivative which for $C^2$ functions corresponds to the largest eigenvalue of the Hessian. The theorem allows one to extend an a.e lower bound on this largest \textquotedblleft eigenvalue\textquotedblright \hspace{.5mm} to a bound holding everywhere. Via the Dirichlet duality theory of Harvey and Lawson, this result has been key to recent progress on the fully non-linear, elliptic Dirchlet problem.
Second, using the Legendre-Fenchel transform we derive a dual interpretation of this largest eigenvalue in terms of convexity of the conjugate function. This dual characterization offers more insight into the nature of this largest eigenvalue and allows for an alternative proof of a bound needed for the theorem.
\end{abstract}
\tableofcontents {}
\section{1. Introduction}
\subsection{1.1 Motivation}
It is known that a convex function $u$ on $\mathbb{R}^n$ is differentiable almost everywhere and has distributional second-order partial derivatives. It is also known that a convex function is twice differentiable almost everywhere in the sense that for a.e. $x\in \mathbb{R}^n$, there exists a symmetric positive semi-definite matrix $D^2f(x)$ such that
$$f(x+h)=f(x)+\langle\nabla f(x), h\rangle + \frac{1}{2} \langle D^2f(x)h,h\rangle+ o(||h||^2). $$
The operator $D^2f$ is called the \textit{second-order Peano derivative}. Note that its existence does not imply the existence of $\nabla f$ in a neighbourhood, so it should not be considered the second derivative of $f$ in the usual sense. This result is due to Alexandrov [1]. See also [5], [6].

In [9], Slodkowski studies uniqueness for a generalized Dirichlet problem in the class of $q-$plurisubharmonic ($q$--psh) functions (for $C^2$ functions on $\mathbb{C}^n$ this is equivalent to the complex Hessian having $n-q$ nonnegative eigenvalues at every point). The problem of uniqueness reduces to showing that the difference of two such functions is $n-1$--psh, which implies that it satisfies a maximum principle, from which uniqueness then follows. Functions of this $q$--psh class can be approximated by a subclass which are convex up to a quadratic polynomial. Because of this it is sufficient to study this smaller class, which given their quasi-convexity, retain some of the nice properties of convex functions. In particular, quasi-convex functions are a.e. twice differentiable, in the above sense. Thus, the second-order behavior of these functions and their difference is known a.e. However, to show that the difference is a member of the above mentioned class, they must satisfy this eigenvalue property everywhere. To this end, Slodkowski introduces a generalized second-order derivative, which is simply the largest eigenvalue of the Hessian for $C^2$ functions, and proves that if this quantity is bounded below almost everywhere in some domain, it is bounded below everywhere in that domain. Using this, he shows that the difference is contained in the desired $n-1$--psh class.

\vspace{.1in}
Following Slodkowski [9, \S 3 ], we define the largest \textquotedblleft eigenvalue\textquotedblright \hspace{.1mm} of a convex function.

\begin{definition} {1.1} \textit{Let $u:\mathbb{R}^n \rightarrow \mathbb{R}$. If $\nabla u(x_0)$ exists, $K(u,x_0)$ is defined by the formula 
$$K(u,x_0) = \limsup_{\epsilon \rightarrow 0} 2\epsilon^{-2} \text{ max } \{ u(x_0 + \epsilon h)  - u(x_0) - \epsilon \langle \nabla u(x_0), h \rangle : h\in S^{n-1} \}$$ otherwise $K(u,x)$ is defined as $+\infty$.}
\end{definition}

This is the generalized second-order derivative that Slodkowski defines. For the sake of context, note that this quantity is a modification to the \textit{second-order upper Peano derivative} of $u$ in the direction of $h$, which is defined as 
$$ \limsup_{\epsilon \rightarrow 0^+}2\epsilon^{-2} (u(x_0+\epsilon h)-u(x_0)-\epsilon \langle \nabla u(x_0),h \rangle).$$ 
Being maximal, this second-order derivative is of particular interest because it corresponds to the largest eigenvalue of the Hessian when defined (which it does, in the above sense, almost everywhere for convex functions), and gives a useful quantity to work with otherwise, especially in the context of Slodkowski's $C^{1,1}$ estimates.

Regarding this quantity $K(u,x)$, Slodkowski shows the following.
\begin{theorem}{1.2}
([9, Cor. 3.5])\textit{ Let $u:\mathbb{R}^n \rightarrow \mathbb{R}$ be a locally convex function in $U\subset \mathbb{R}^n$, such that $K(u,x)\geq M$ for almost every $x\in U$. Then $K(u,x)\geq M$ for all $x\in U$.}
\end{theorem}

As mentioned above, the recent work of Harvey and Lawson on the Dirichlet problem was one of our motivations for studying this quantity $K(u,x)$ and Slodkowski's proof of the above result. In [4] they study fully non-linear degenerate elliptic equations of the form 
\begin{align}
F(\text{Hess}(u))=0 \text{ on } \Omega\\
u=\phi \text{ on } \partial \Omega.
\end{align}
Given certain convexity assumptions on the boundary, they establish the existence and uniqueness of continuous solutions using their new Dirichlet duality theory. The work of Slodkowski was \textquotedblleft an inspiration\textquotedblright \hspace{.1mm} for that paper, and in particular Theorem 1.2 is the \textquotedblleft deepest ingredient\textquotedblright \hspace{.25mm} of their proof of uniqueness of viscosity solutions of (1) [4, p. 398]. These existence and uniqueness results apply to many important problems including all branches of the homogeneous Monge-Amp\`{e}re equation, all branches of the special Lagrangian potential equation, and equations appearing naturally in Lagrangian and calibrated geometry.

Given the usefulness of this generalized derivative and the above result to recent progress on important problems, it makes sense to better understand both the derivative and the proof of the theorem. The proof is fairly difficult and very geometric so here an illustrated exposition is provided. The quantity $K(u,x)$ is then studied further for convex $u$. In particular, the Legendre--Fenchel transform is applied to give a simple alternative characterization of $K(u,x)$ in terms of the convexity of the dual function $u^*$ to $u$. This allows for an alternative proof to a key proposition needed to prove Slodkowski's theorem. Altogether, there are now three ways to view this generalized derivative $K(u,x)$: analytically (Definition 1.1), geometrically (Proposition 1.6), and dually (Theorem 1.9).

\subsection{1.2 Summary}
Theorem 1.2 follows immediately from the following theorem, the proof of which is the main focus of the first part of this paper.
\begin{theorem}{1.3} 
([9, Thm. 3.2])\textit{ Let $u$ be convex near $x_0\in \mathbb{R}^{n}.$ Assume that $K(u,x_0)=k_0$ is finite. Then for every $k>k_0$ the set $\{x :K(u,x)<k\}$ is Borel and its lower density at $x_0$ is not less than $\left(\frac{k-k_0}{2k}\right)^n.$}
\end{theorem}
Lower density is defined as follows.
\begin{definition} {1.4} \textit{The lower density of a Lebesgue measurable set $Z\subset \mathbb{R}^{n}$ at $x_0 \in \mathbb{R}^{n}$ is the number
$$ \liminf_{\varepsilon \rightarrow 0} \dfrac{m_n \left( Z\cap B(x_0,\varepsilon) \right)}{m_{n}\left(B(x_0,\varepsilon) \right)},$$
 where $m_{n}$ denotes the $n$-dimensional Lebesgue measure.}
\end{definition}

Slodkowski's proof of Theorem 1.3 divides naturally into two parts. First, an equivalent geometric characterization of a bound on $K(u,x)$ is given in terms of spheres tangent to the graph of $u$. This is the content of the following definition and proposition.

For $c=(c_1,\dots,c_{n+1})\in \mathbb{R}^n$, let $S(c,r)$ denote the $n-$sphere with center $c$ and radius $r$, and $B(c,r)$ denote the open $n+1$-disk of radius $r$ centered at $c$.
\begin{definition} {1.5}
\textit{The sphere $S(c,r)$ is a sphere of support from above at $y=(x_0,u(x_0))$ if $y\in S(c,r)$, $B(c,r) \cap \text{graph}(u)=\emptyset$ and $c_{n+1}>u(P(c))$, where $P$ denotes the orthogonal projection of $\mathbb{R}^{n+1}$ onto $\mathbb{R}^n$.}
\end{definition}

Thus, $S(c,r)$ can be visualized as a ball resting on a \textquotedblleft surface\textquotedblright \hspace{1mm} that is the graph of $u$, and such that $(x_0,u(x_0))$ is one of its resting points.

\begin{proposition}{1.6}
([9, Prop. 3.3]) \textit{Let $U\subset \mathbb{R}^{n}$ be open and $u:U\rightarrow \mathbb{R}$ be convex. Assume that $u$ has gradient at $x$}.\\

\textit{(i) If $u$ has second-order Peano derivatives at $x$, then $K(u,x)$ is equal to the norm (i.e. the largest eigenvalue) of the real Hessian of $u$ at $x$}.\\

\textit{(ii) If $K(u,x)$ is finite, then for every $K>K(u,x)$ there is $\varepsilon > 0$ such that $u(x+h)-u(x)-\langle\nabla u(x),h\rangle \hspace{2mm}\leq \dfrac{1}{2}K|h|^{2}.$}\\

\textit{(iii) If there is a sphere $S(c,r), r>0$ which supports the graph of $u$ from the above at $(x,u(x))$, then 
\begin{align}
K(u,x)\leq \dfrac{(1+|\nabla u(x)|^{2})^{\frac{3}{2}}}{r}.
\end{align}}
\end{proposition}
Parts (ii) and (iii) give the above mentioned equivalence between a bound on $K(u,x)$ and a sphere of support to the graph of a corresponding radius at $(x,u(x))$. See section 2.2 for a more detailed explanation.\\

The second part of the proof then uses this alternative characterization of $K(u,x)$ to obtain a density result, which is essentially the statement of the theorem in terms of spheres of support as opposed to $K(u,x)$. This is the content of the following lemma.

\begin{lemma}{1.7} 
([9, Lemma 3.4]) \textit{Let $u$ be a non-negative convex function in $B(0,d)\subset \mathbb{R}^{n}$, $d>0$, such that $u(0)=0$ and $\nabla u(0)=0$. Let $R>0$ and assume that the closed ball \={B}$(c,R)$, $c=(0,...,0,R)\in \mathbb{R}^{n+1}$, intersects the graph of $u$ only at $0\in \mathbb{R}^{n+1}$. Let $X_{r}$, $0<r<R$ denote the set of all $x\in B(0,d)\subset \mathbb{R}^{n}$ such that there exists a sphere of radius $r$ supporting the graph of $u$ from above at $(x,u(x))$. Then the lower density of $X_{r}$ at 0 is not less than $((R-r)/2R)^{n}$.}
\end{lemma}
As will be seen in more detail in section 2, there is an inverse relationship between the bound on $K(u,x)$ and the radius of the sphere of support to the graph of $u$ at $(x,u(x))$. This will explain the similarity between the lower bound on density given in the lemma and the one in the theorem.\\

The geometric characterization of $K(u,x)$ is key to proving Theorem 1.3 and helpful in understanding what quality this generalized derivative captures about the function $u$ and its graph. Since the results here concern functions that are at least locally convex, it is natural to study them via the Legendre--Fenchel transform, the classical transform of convex analysis. By definition, the set of points above the graph of a convex function (epigraph) is a convex set. Any convex set in $\mathbb{R}^n$ can be defined entirely by a family of supporting hyperplanes. Thus, since the epigraph of $u$ completely determines the graph of $u$, which in turn completely determines $u$, this family of hyperplanes can be considered an alternative description or parametrization of $u$. This is essentially how the transform of $u$ (or dual function to $u$) $u^*$ is defined. Each point $p\in \mathbb{R}^n$ defines a collection of hyperplanes (via gradient), and $u^*$ specifies a point $u^*(p)\in\mathbb{R}$, such that $(0,...,0,-u^*(p))\in \mathbb{R}^n$ lies on the one hyperplane of this collection which supports the epigraph (or graph) of $u$.

Interestingly, under the Legendre--Fenchel transform, differentiability properties of $u$ correspond to convexity properties of $u^*$. Two classic examples of this are the following.
\begin{proposition}{1.8}
\textit{Let $f:\mathbb{R}^n\rightarrow \mathbb{R}$. Then}\\
\textit{(i) $f \text{ is strictly convex if and only if } f^* \text{ is differentiable} .$}\\
\textit{(ii) $f \text{ is strongly convex with modulus } c \text{ if and only if } f^* \text{ is differentiable and }  \nabla f^* \text{ is Lipschitz continuous }$}\\  \textit{with constant} $\frac{1}{c}.$
\end{proposition}
Given that $K(u,x)$ is a (local) differentiability property of $u$, it seems there should be an appropriate (local) convexity property corresponding to $u^*$. In section 3 we prove the following result.
\begin{theorem}{1.9}
\textit{Let $f:\mathbb{R}^n \rightarrow \mathbb{R}$ be convex. If $K(f,x_0)=k_0<k$ then $f^*$ is quadratically convex at $y_0=\nabla f(x_0)$ with modulus $\frac{1}{k}$. Conversely, if $u^*$ is quadratically convex with modulus $\frac 1k$, then $K(f,x_0)=k_0\leq k$}.
\end{theorem}
Quadratically convex at $y_0$, which is defined in section 3, is a more local form of convexity than the two types of convexity referred to in Proposition 1.8. This dual characterization of $K(u,x)$ allows for an alternative proof of Proposition 1.6. Using quadratics to define different types of convexity is standard (e.g. quasi-convexity, strong convexity). See section 3 for definitions of all these terms and a more detailed discussion.

In Slodkowski's proof quadratics arise naturally via the definition of $K(u,x)$, and from this, spheres. The geometric properties of spheres make certain arguments very clear (see proof of Lemma 1.7), however some manipulations and calculations are simpler with quadratics, given their constant second-order behavior. For example, in [5] Harvey and Lawson provide an alternative proof of Slodkoski's lemma (as well as Alexandrov's theorem stated above) via a generalization by using quadratics instead of spheres. Their proof is modelled off of Slodkowski's, and they obtain their result for the larger class of quasi-convex functions. Instead of spheres of support, they use the notion of upper contact jets, where given $p\in \mathbb{R}^n,$ and $A$ a real symmetric $n\times n$ matrix, $(p,A)$ is an \textit{upper contact jet for $u$ at $x$} if there exists a neighbourhood of $x$ such that 
$$u(y)\leq u(x) +\langle p,y-x \rangle + \frac{1}{2}\langle A(y-x), y-x\rangle.$$
Slodkowski's result then corresponds to $A=\lambda I$.

\subsection{1.4 Organization}
Section 2 contains the exposition of Slodkowski's proof of Theorem 1.3: \S 2.1 gives an overview of the proof, \S 2.2 a slight variation of Slodkowski's proof to Proposition 1.6 (the generalized $C^{1,1}$ estimate), \S 2.3 an expanded and illustrated version of Slodkowski's proof to Lemma 1.7, and \S 2.4 combines these for the proof of the theorem.\\
Section 3 studies $K(u,x)$ from the dual perspective: \S 3.1 recalls some basic convex analysis, including Legendre--Fenchel duality, \S 3.2 provides an equivalent interpretation of $K(u,x)$ in terms of the dual function to $u$, and uses this for an alternative proof of the $C^{1,1}$ estimate.\\
The Appendix considers Lipschitz continuity of the gradient and the geometric interpretation of $K(u,x)$: \S A.1 demonstrates $K(u,x)$ is bounded by the Lipschitz constant when $u$ is $C^{1,1}$, \S A.2 gives an example of a function with a sphere of support that is not $C^{1,1}$ on any neighbourhood, \S A.3 compares $K(u,x)$ to the classical notion of an osculating circle to a plane curve and gives an extension of this to higher dimensions, \S A.4 relates the radius of a sphere of support to a function to that of the radius of a supporting sphere to its dual.
\section{2. Exposition of Slodkowski's proof}
\subsection{2.1 Overview}
Theorem 1.3 is concerned with the set of points (near $x_0$) such that $K(u,x)<k$, for some fixed $k>k_0= K(u,x_0)$. However this set may be difficult to study directly given that the only information available about $u$ is that it is continuous (bounded and convex) on some neighbourhood of $x_0$ and $K(u,x_0)=k_0<\infty$. In particular, knowing the value of $K(u,x)$ at a given point does not immediately suggest anything about its value nearby.
Thus, the first step towards a better understanding of this set of points is an alternative characterization of what it means for $K(u,x)$ to bounded at some point.

If at the point $x$, $K(u,x_0)<\infty$ this is equivalent to a (local) sphere of support from above to the graph of $u$ at $(x,u(x))$. This is precisely what Proposition 1.6 (ii) and (iii) states. (ii) implies the existence (locally) of a quadratic function tangent to the graph of $u$ at $(x,u(x))$ which majorizes $u$ on some neighbourhood, and this in turn implies the (local) existence of a sphere of support to the graph of $u$ at $(x,u(x))$. The content of (iii) is clear.

With this alternative geometric characterization in hand, Lemma 1.7 then proves the theorem in terms of these spheres of support. To accomplish this another change in perspective is needed, which takes further advantage of this more geometric interpretation of $K(u,x)$. Instead of looking at points $x$ in the domain of $u$ such that there exists a sphere of support to the graph of $u$ at $(x,u(x))$, it is better to consider for each point $x$ in domain of $u$ an $n-$sphere (of fixed radius) in $\mathbb{R}^{n+1}$ above the graph of $u$ with center $c\in \mathbb{R}^{n+1}$ such that $P(c)=x$, where $P:\mathbb{R}^{n+1}\rightarrow \mathbb{R}^{n}$ is the projection map. If we lower this sphere down towards $x$ it will of course eventually intersect the graph of $u$. Since $u$ is continuous, it is not difficult to show that on a small enough neighbourhood these spheres will come down on a closed part of the graph of $u$ and thus there will be an initial point of contact. This sphere is by definition a sphere of support to the graph of $u$ at that point. The next step is to show that for every $\epsilon$ neighbourhood of 0 ($x_0=0$ for Lemma 1.7) there is a corresponding $\delta=\delta(\epsilon)$ such that the spheres above the points in $B(0,\delta)$ are spheres of support to the graph at points $(x,u(x))$, where $x\in B(0,\epsilon)$. Now $B(0,\delta)$ is a much nicer set to work with then $X_r\cap B(0,\epsilon)$, and these two sets can be related by a few simple Lipschitz maps. Since Lipschitz maps behave nicely with respect to measures, this allows us to place a lower bound on the measure $m(X_r\cap B(0,\epsilon))$ for each epsilon. A limiting argument is then used to obtain the lower bound on the lower density at 0.

Proposition 1.6 and Lemma 1.7 can then be combined to give Theorem 1.3. A sketch of the proof is as follows. Start with a point $x_0$ where $K(u,x_0)$ is finite (hypothesis of Theorem 1.3), and choose any $k>K(u,x_0)$. Note it can be assumed without loss of generality that $x_0=0$, $u(0)=0$, and $\nabla u(0)=0$ (see section 2.3 for details). Then apply Proposition 1.6 (ii), which locally gives a sphere of support of radius $1/k$ at $(x_0,u(x_0))$. Now, apply Lemma 1.7 to get a lower bound on the density of $X_r$, $r<1/k$, at $x_0$. Next, apply Proposition 1.6 (iii) to convert this into a statement about the density of $X'_k$, where
\begin{align*}
X'_k \equiv \{ x\in \text{dom}(u) | K(u,x)<k \}.
\end{align*}
This last step is accomplished by using the continuity of the gradient to show that in a small enough neighbourhood $ X_r \subset X'_k $. More explicitly, $x\in X_r$ implies $K(u,x)\leq r^{-1}(1+|\nabla u(x)|^2)^{3/2}$ and $\nabla u(x_0)=0$, so by continuity of the gradient of convex functions and since $k>1/r$, $\nabla u(x)$ will eventually be small enough so that $r^{-1}(1+|\nabla u(x)|^2)^{3/2}<k$. Thus, for $x\in X_r$, $K(u,x)<1/k$. This gives the theorem by choosing $R$ arbitrarily close to $1/k_0$ and $r$ arbitrarily close to $1/k$ (see section 2.4 for a detailed proof).
\subsection{2.2. The generalized $C^{1,1}$ estimate}
\begin{figure}[H]
\centering
\begin{tikzpicture}

\draw[thick, ->] (0,0) -- (5,0); %axes
\draw[thick, ->] (0,0) -- (0,4);

\draw[thick] (0,.5) parabola (4,2.5) node[anchor=south west] {$u$}; %u
\draw[thick,dashed] (-.05,2.4) -- (.05,2.4);

\draw[thick, dashed] (2,2.4) -- (2,0) node[anchor=north] {};
\draw[thick, dashed] (2.65,1.4) -- (2.65,-.005) node[anchor=north west ] {};
\draw[thick, dashed] (3.1,1.9) -- (3.1,0);
\draw[thick, dashed] (3.1,.05) -- (3.1,0);

\draw[thick, gray] (2,2.4) -- (3.1,1.9) node[anchor=south east] {$r$}; %radius

\node at (3.35,.06) [label=below:\small{$x+h$}]{} ; %x-axis points
\node at (2.65,-.025) [label=below:\small{$x$}]{} ;
\node at (2,-.025) [label=below:\small{$c$}]{} ;
\node at (-.25,2.7) [label=below:\small{$t$}]{} ;
\node at (0,4.2) [label=left:\small{$\mathbb{R}^n$}]{} ;
\node at (5,0) [label=below:\small{$\mathbb{R}$}]{} ;

\coordinate (center) at (2,2.4); %lower hemisphere
\cercle{center}{1.2cm}{180}{180};
\node at (1.25,2.25) [label=below:\small{$d$}]{} ;

\end{tikzpicture}

\caption{$d:B(c,r)\rightarrow \mathbb{R}$}
\end{figure}
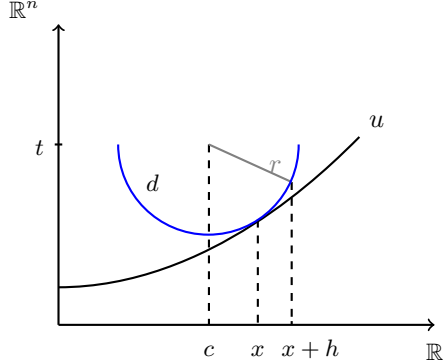
In this subsection we provide an alternative proof to Proposition 1.6 (iii). The main idea is as follows: given a sphere of support of radius $r$ to the graph of $u$ at the point $(x,u(x))$, the lower hemisphere of this sphere defines the graph of a smooth convex function that agrees up to first order with $u$ at $x$ and majorizes $u$ elsewhere. Denote this function by $d$. It immediately follows that $K(u,x)\leq K(d,x)$, and the rest of the proof consists in computing $K(d,x)$, which is equal to the largest eigenvalue of $d$ because $d$ is smooth \\
\textit{Proof of Proposition 1.6 (iii).} Assume that the sphere $S((c,t),r), \hspace{1mm} c\in \mathbb{R}^{n}$ supports the graph of $u$ from the above at $(x,u(x_0))$ and that $u$ is differentiable at $x_0$.
Define $d:B(c,r)\rightarrow \mathbb{R}$ to be the function whose graph is the lower open hemisphere of $S((c,t),r).$ 
Recall the definition for $K(f,x_0)$ :
$$K(u,x_0) := \limsup_{\epsilon \rightarrow 0} 2\epsilon^{-2} \text{ max } \{ u(x_0 + \epsilon h)  - u(x_0) - \epsilon \langle \nabla u(x_0), h \rangle : |h| =1 \}.$$\\
Clearly, since $d(x_0)=u(x_0)$ and $\nabla d(x_0)=\nabla u(x_0)$,
$$K(u,x_0)\leq K(d,x_0).$$
Since $d$ is smooth,
\begin{align*}
K(d,x_0) =& \limsup_{\epsilon \rightarrow 0} 2\epsilon^{-2} \text{ max } \{ d(x_0 + \epsilon h)  - d(x_0) - \epsilon \langle \nabla d(x_0), h \rangle : |h| =1 \} \\
=& \limsup_{\epsilon \rightarrow 0} 2\epsilon^{-2} \text{ max } \{ \frac{1}{2} \langle \nabla^2 d(x_0+ \gamma_{\epsilon,h} \epsilon h)\epsilon h, \epsilon h \rangle : |h| =1 \}, \hspace{5mm}0<\gamma_{\epsilon,h} <1 \\
=& \limsup_{\epsilon \rightarrow 0}  \text{ max } \{ \langle \nabla^2 d(x_0+ \gamma_{\epsilon,h} \epsilon h) h,  h \rangle : |h| =1 \}, \hspace{5mm}0<\gamma_{\epsilon,h} <1 \\
=& \text{ max } \{ \langle \nabla^2 d(x_0) h,  h \rangle : |h| =1\} \text{ by continuity and compactness.}\\
=& \lambda_\text{max}, \hspace{5mm} \text{ maximum eigenvalue of } \nabla^2 d(x_0)
\end{align*} 
Thus, now we show that $$\lambda_\text{max}=\dfrac{(1+(\nabla u(x_0))^2)^{\frac{3}{2}}}{r}.$$\\
The equation for $d$, the sphere of radius $r$ centered at $(c,t)$, where $c\in \mathbb{R}^n$ and $t\in \mathbb{R}$, is $$d(x)= t -\sqrt{r^2 -|c-x|^2}.$$ Without loss of generality we may assume that the sphere of support is centered at the origin and $x_0$ has just first component non-zero, as otherwise we could always shift and then rotate without affecting the second-order behavior. In other words, assume $(c,t)=0 \in \mathbb{R}^{n+1}$ and $x_0=(s_1,...,s_n)=(s,0,...,0)\in \mathbb{R}^n$. Then $d(x)= -\sqrt{r^2 -|x|^2}.$ 

Let $$w(x):= \frac{1}{r^2-s^2} \left( |x-x_0|^2 + 2\langle x-x_0,x_0 \rangle  \right).$$ Since $$|x|^2=\langle x,x \rangle= \langle (x-x_0)+x_0, (x-x_0)+x_0 \rangle= |x-x_0|^2 +2 \langle x_0, x-x_0 \rangle +|x_0|^2$$ and $|x_0|^2=s^2$, we can write $d(x)$ as 
$$d(x)=-\sqrt{r^2-s^2}\sqrt{1-w(x)}.$$
Now expanding $\sqrt{1-w(x)}$ as a series and dropping the terms of order higher than two (as they will have 0 Hessian at $x_0$), $$d(x)\approx -\sqrt{r^2-s^2} \left( 1-\frac{w(x)}{2} -\frac{w(x)^2}{8}  \right).$$ This can be further reduced to $$d(x)\approx -\sqrt{r^2-s^2} \left( 1-\frac{w(x)}{2} -\frac{1}{8}\left( \frac{2\langle x-x_0,x_0\rangle }{r^2-s^2}\right)^2  \right),$$ since we are only concerned with the expression for $d$, modulo powers higher than two.

Thus, $d(x)$ has been replaced by a diagonal quadratic form and straightforward computations give $$\nabla d(x_0)= \frac{x_0}{\sqrt{r^2-s^2}},$$ and $$\nabla^2d(x_0)= \frac{1}{\sqrt{r^2-s^2}} I + \frac{s}{(r^2-s^2)^{3/2}} A,$$ where $I$ is the $n\times n$ identity matrix and $A$ is the $n\times n$ matrix with first row $x_0=(s,0,...,0)$ and zeros elsewhere. Since $$\frac{s^2}{(r^2-s^2)^{\frac{3}{2}}}>0,$$
it follows immediately that $$\lambda_{max}=\frac{1}{\sqrt{r^2-s^2}} + \frac{s^2}{(r^2-s^2)^{3/2}}=\frac{r^2}{(r^2-s^2)^{\frac{3}{2}}}.$$
Furthermore, the vector $(x_0,u(x_0))$ is of length $r$, proportional to the upward pointing unit normal to the graph of $u$ at $(x_0,u(x_0))$, which is equal to 
$$\frac{1}{\sqrt{1+|\nabla u(x_0)|^2}}(-\nabla u(x_0),1).$$
Scaling by $r$, we obtain
$$x_0=\frac{-r}{\sqrt{1+|\nabla u(x_0)|^2}}\nabla u(x_0).$$
Giving $$x_0=\dfrac{r \nabla u(x_0)}{\sqrt{1+|\nabla u(x_0)|^2}}, \hspace{3mm} s^2=|x_0|^2=\dfrac{r^2|\nabla u(x_0^2)}{1+|\nabla u(x_0)|^2}.$$
Therefore,
$$\lambda_{max}=\dfrac{(1+|\nabla u(x_0)|^2)^\frac{3}{2}}{r}.$$
$\hspace{165mm} \square$

We state explicitly the following interesting result on \textquotedblleft lower hemisphere functions\textquotedblright, i.e. functions on a disc $D \subset \mathbb{R}^n$ defined by the lower hemisphere of an $n-$ sphere in $\mathbb{R}^{n+1}$. The proof follows immediately from the above proof, by looking at the expression for the Hessian.
\begin{proposition}{2.1}
Let $d:D\rightarrow \mathbb{R}$ be a lower hemisphere function defined on a disc $D\subset \mathbb{R}^n$ and $\tilde{x} \in D$. If $\nabla d(\tilde{x}) \neq 0$, then $\nabla d(\tilde{x}) \neq 0$ is an eigenvector of $\nabla ^2 d(\tilde{x})$ corresponding to the largest eigenvalue.
\end{proposition}
\begin{proof}
Without loss of generality we may assume that the lower hemisphere and thus $D$ are centered at the origin and $\tilde{x}$ has only first coordinate non-zero, $\tilde{x} =(s,0,...,0)$. Then, as shown above, the Hessian of $d$ at $\tilde{x}$ is a diagonal $n\times n$ matrix of the form
$$\nabla ^2 d(\tilde{x})= \textnormal{diag}\left(\frac{1}{\sqrt{r^2-s^2}}+\frac{s^2}{(r^2-s^2)^{\frac{3}{2}}},\frac{1}{\sqrt{r^2-s^2}},...,\frac{1}{\sqrt{r^2-s^2}}\right).$$
Thus, $(1,0,...,0)$ is an eigenvector corresponding to the largest eigenvalue. As calculated above,
$$\nabla d(\tilde{x})= \frac{\tilde{x}}{\sqrt{r^2-s^2}}= \frac{s}{\sqrt{r^2-s^2}} (1,0,...,0) ,$$ so clearly $\nabla d(\tilde{x})$ is also an eigenvector corresponding to the largest eigenvalue.

\end{proof}

\subsection{2.3 The density lemma} 
If at the point $x_0=0$ there is a sphere of support of radius $R$, Lemma 1.7 provides a lower bound on the lower density of the set $X_r$ of points with sphere of support of a radius $r<R$. Note that without loss of generality it may be assumed that $x_0=0$, $u(0)=0$, and $\nabla u(0)=0$, since any convex function \textit{\~{u}} can always be adjusted by a constant and linear term so that this is true without affecting the 2nd-order behaviour of \textit{\~{u}}.

As mentioned in section 2.1, Lemma 1.7 is proved by looking not directly at $X_r$ but at small neighbourhoods of 0 that are the projection of the set of centers of spheres of support to the graph of $u$ on shrinking neighbourhoods. For each $\epsilon>0$ a $\delta=\delta(\epsilon)$ is needed so that $B(0,\delta)$ is contained in the projection onto $\mathbb{R}^n$ of the set of centers of spheres of support to the graph of $u$ restricted to an epsilon neighbourhood. Since the only information about $u$ is that there is a sphere of support at 0, this is what is used to construct $\epsilon$ and $\delta$. 
More specifically, the appropriate $\epsilon$'s and $\delta$'s are found by constructing a family of convex functions that are identical to $u$ on a neighbourhood of $0$, but greater and simpler outside this neighbourhood. This allows one to fully utilize the only initial information given. Using this family of simple functions and basic geometry, three key set inclusions are obtained, which essentially relate $B(0,\delta(\epsilon))$ to $X_r\cap B(0,\epsilon)$. Then using Lipschitz maps to relate these sets and by applying properties of Lipschitz functions on measure, the lower density bound is shown. This whole construction is crucial because it provides a much simpler approach to studying the possibly very complex set $X_r$. The following is the proof given by Slodkowski.

\begin{proof} \textit{of Lemma 1.7.}\\
The number $r \in (0,R)$ will be kept fixed so let $X \equiv X_{r}$. Define $$Z= \{(x,u(x))\in \mathbb{R}^{n+1}:x\in X\}.$$ It is clear that $Z\cap(\text{\={B}}(0,d')\times\mathbb{R})$ is compact for every $d'<d$, thus $X\cap(\text{\={B}}(0,d')\times\mathbb{R})$ is also compact, as it is the orthogonal projection $P:\mathbb{R}^{n+1}\rightarrow \mathbb{R}^{n}$ of $Z$. Since compact sets are Lebesgue measurable, the notion of lower density is applicable to both $X$ and $Z$.

It is more convenient to first estimate the density of $Z$ at 0 with respect to Hausdorff measure, and then use the properties of Lipschitz functions on measure to obtain bounds on the density of $X$. To accomplish this a family of convex functions, built from the initial sphere of support of radius $R$ at 0, which modify $u$ outside a small neighbourhood of 0 will be constructed. As mentioned above, these functions will be identical to $u$ on a neighbourhood of $0$ and very simple outside this neighbourhood. These functions will enable us to find a corresponding $\delta=\delta(\epsilon) $ neighbourhood for each $\epsilon$ so that $x\in B(0,\delta)$ implies that $x=P(c)$, where $c\in \mathbb{R}^n$ is the center of a sphere of support to $(x',u(x'))$, for some $x'\in B(0,\epsilon) \cap X_r$.\\
\textbf{\textit{Step One.}} A family of convex functions is constructed which will let us find an appropriate $\delta(\epsilon)$, as explained above. For each $\alpha$ such that $0<\alpha<\frac{1}{2}\arcsin (\frac{d}{R})$, define the function 
$$
v_{\alpha}:B(0,R)\rightarrow[0,\infty),
$$
as follows. First, define
\begin{align}
Y=\{y\in \mathbb{R}^{n+1}:|y-c|=R, \varangle (y-c, 0-c)=2\alpha\},
\end{align} where $c=(0,...,0,R)\in \mathbb{R}^{n+1}$ is the center of the sphere of support to $u$ at $(0,u(0))$.
%(**Then $|y-c^{*}| |c^{*}|\cos 2\alpha=(y-c^{*}, -c^{*})=-R_{N+1}+R^{2}$ (delete?)**) %\\
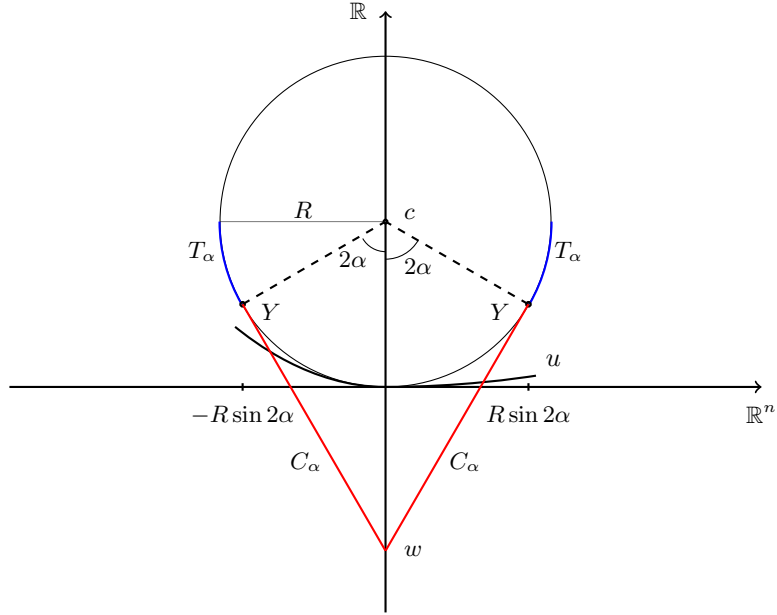
\begin{figure}[!h]
\centering
\begin{tikzpicture}

\draw[thick, ->] (-5,0) -- (5,0); %axes
\draw[thick, ->] (0,-3) -- (0,5);

\draw[thick, dashed] (1.9,1.1) -- (0,2.2); %radii
\draw[thick, dashed] (-1.9,1.1) -- (0,2.2);

\draw[thick] (0,0) parabola (2,.15) node[anchor=south west] {$u$}; % convex function u
\draw[thick] (0,0) parabola (-2,.8) ;
\draw[thin,gray] (0,2.2)--(-2.2,2.2);

\draw (0,2.2) circle (2.2cm) ; % supporting sphere 

\draw[thick] (0,2.2) circle (0.025cm) ; %center of supporting sphere
\node at (0,2.3) [label=right:\small{$c$}]{} ;

\coordinate (center) at (0,2.2); %T_alpha

\cercle{center}{2.2cm}{0}{-30};
\cercle{center}{2.2cm}{180}{30};

\node at (0,1.7) [label=left:\small{$2\alpha$}]{} ; %2alpha
\node at (0,1.6) [label=right:\small{$2\alpha$}]{} ;

\node at (-1.9,1) [label=right:\small{$Y$}]{} ; %Y
\node at (1.9,1) [label=left:\small{$Y$}]{} ;
\draw[thick] (-1.9,1.1) circle (0.03cm) ;
\draw[thick] (1.9,1.1) circle (0.03cm) ;

\draw (0,1.7) arc (270:330:.5cm);
\draw (0,1.8) arc (270:210:.35cm);

\draw[thick, red] (0,-2.18) -- (1.9,1.1); %C_alpha
\draw[thick, red] (0,-2.18) -- (-1.9,1.1);

\node at (0,-2.18) [label=right:\small{$w$}]{} ;

\node at (2,1.8) [label=right:\small{$T_{\alpha}$}]{} ; %label for T_alpha
\node at (-2,1.8) [label=left:\small{$T_{\alpha}$}]{} ;

\node at (.6,-1) [label=right:\small{$C_{\alpha}$}]{} ; %label for C_alpha
\node at (-.6,-1) [label=left:\small{$C_{\alpha}$}]{} ;

\node at (0,5) [label=left:\small{$\mathbb{R}$}]{} ;
\node at (5,0) [label=below:\small{$\mathbb{R}^n$}]{} ;

\node at (-1.1,2) [label=above:\small{$R$}]{} ;

\draw[thick] (1.9,-.05) -- (1.9,.05);%tick marks
\draw[thick] (-1.9,-.05) -- (-1.9,.05);

\node at (1.9,0) [label=below:\small{$R \sin 2\alpha$}]{} ;
\node at (-1.9,0) [label=below:\small{$-R \sin 2\alpha$}]{} ;
\end{tikzpicture}

\caption{Construction of Auxiliary Convex Functions}
\end{figure}
$Y$ forms a \textquotedblleft ring \textquotedblright on $S(c,R)$, and clearly the projection of $Y$, $P(Y)$, onto $\mathbb{R}^{n}$ is the $n-1$ sphere of radius $R\sin 2\alpha$, centered at 0. Next, let $C_{\alpha}$ denote the union of all closed segments $\overline{\text{wy}}$ with one endpoint $w$ on the axis $0\times \mathbb{R}\subset \mathbb{R}^{n+1}$ and tangent to the sphere $S(c,R)$ at the other endpoint y, where $y\in Y$. Note that w is independent of which $y\in Y$ that is being used. $C_{\alpha}$ is simply a finite cone with vertex $w$ and base $Y$, tangent to $S(c,R)$ along $Y$. See Figure 2.\\
Define now 
\begin{align}
T_{\alpha}= \{y\in S(c,R): R(1-\cos 2\alpha)\leq y_{n+1}< R\}.
\end{align}
$T_{\alpha}$ can be visualized as a \textquotedblleft strip\textquotedblright \hspace{.25mm} of $S(c,r)$, and note that $T_{\alpha}\cap C_{\alpha}=Y$ and that $T_{\alpha}\cup C_{\alpha}$ defines a convex function $k_{\alpha}:B(0,R)\rightarrow R$.\\
For $0<\alpha<\frac{1}{2}\arcsin (\frac{d}{R})$, define
\begin{align*}
v_{\alpha}=
\begin{cases}
\text{max}(u(x),k_{\alpha}(x)),\hspace{5mm} |x|<R\sin 2\alpha\\
k_{\alpha}(x), \hspace{23mm} R\sin 2\alpha \leq |x|< R.
\end{cases}
\end{align*}
Note that $u$ is only defined on $B(0,d)$ and $R\sin 2\alpha <d<R$, so that is why $v_{\alpha}$ is defined this way. 
It is clear that
\begin{align}
v_{\alpha}(x)\geq u(x), \text{ for } |x|<d.
\end{align}

Observe that $v_{\alpha}$ is locally convex on the set $|x|\neq R\sin 2\alpha$ since for $|x|> R\sin 2\alpha,  v_{\alpha}=k_{\alpha}(x)$, which is convex, and for $|x|< R\sin 2\alpha, v_{\alpha}$ is the maximum of two convex functions which is convex. If $|x|=R\sin 2\alpha$, then $(x,v_{\alpha}(x))\in Y\subset S(c,r)$ Since $S(c,r)$ lies above the graph of $u$, so $k_{\alpha}|_{Y} > u|_{Y}$. Thus near $Y, v_{\alpha}\equiv k_{\alpha}$, and so $v_{\alpha}$ is locally convex in $B(0,R)$, which implies that $v_{\alpha}$ is convex.\\
\textbf{\textit{Step Two.}} For any convex function the following Lipschitz map can be constructed. This will let us relate the possibly complex set, $X$, to the disk $B(0,\delta(\epsilon))$. Given a convex function $v:B(0,R)\rightarrow \mathbb{R}$. Let $E(v)=\{(x,t)\in \mathbb{R}^{n+1}: t>v(x) \}$ denote the strict epigraph of $v$, and define $Z^{v}$ as the set of all $y=(x,v(x))$, where $|x|< R$, and such that for some $c'\in \mathbb{R}^{n+1}$, $B(c',r)\subset E(v)$ and $y\in S(c',r)$, where $r<R$, as defined earlier.

Note that if $y=(x,v(x))\in Z^{v}$, then the graph($v$) has a unique supporting hyperplane at $y$ (since any such hyperplane is tangent to $S(c',r)$), and thus $c'$ is uniquely determined by $y$.

Now consider the map $\gamma^{v}:Z^{v}\rightarrow \mathbb{R}^{n+1}$, where $\gamma^{v}(y)=c'$. This map is Lipschitz with constant one. To see this, let $y_{1},y_{2}\in Z^{v}$ and $c'_{i}=\gamma^{v}(y_{i}), i=1,2.$ The set $E(v)$ is convex (by definition since $v$ is convex), and so it contains $W := \text{co}(B(c_{1},r)\cup B(c_{1},r))$, where co() denotes the convex hull. In particular, $W\cap$ graph($v$)=$\emptyset$. Since $y_{i}\in S(c_{i},r)\cap \text{graph}(v)$, $y_{i}\in S(c'_{i},r)\setminus W, i=1,2$. Thus, $y_{1}$ and $y_{2}$ do not belong to, and are separated by, the open region between two hyperplanes which are orthogonal to the segment $\overline{c'_{1}c'_{2}}$ and pass through its ends. Therefore $|c'_{1}-c'_{2}|\leq |y_{1}-y_{2}|$. The importance of this map will be seen below, where combined with $u$ and the projection map $P$ it allows the set of interest in $\mathbb{R}^n$ to be related to a small disk.\\
\textbf{\textit{Step Three.}} Three key set inclusions are established. Along with step two this will allow on small neighborhoods the measure of $X$ to be bounded from below by the volume of small $n-$ balls. Using the notation above, let $Z^{\alpha}$ and $\gamma^{\alpha}$ denote the set $Z^{v}$ and map $\gamma^{v}$, respectively, for $v=v_{\alpha}$, where $0<\alpha<\frac{1}{2}\arcsin (\frac{d}{R})$.\\
Consider the set 
\begin{align}
U_{\alpha}=\text{graph}(v_{\alpha})\setminus (C_{\alpha}\cup T_{\alpha}).
\end{align}
Note that this is a subset of the graph of $u$.
For $\alpha \in (0,\frac{1}{2}\arcsin(\frac{d}{R}))$, we have the following three inclusions:
\begin{align}
&P(U_{\alpha})\subset B(0,R\sin 2\alpha)\\
&Z^{\alpha}\cap U_{\alpha} \subset Z \cap U_{\alpha}\\
&B_{N}(0,\delta)\subset P\gamma^{\alpha}(Z^{\alpha}\cap U_{\alpha}), \text{where }\delta=(R-r)\tan \alpha.
\end{align}

The first inclusion follows directly from the definition of $U_{\alpha}$: $|x| \geq R \sin 2\alpha \Rightarrow v_{\alpha}(x)\in T_{\alpha}$.

By (5), $Z^{\alpha}\cap \text{graph}(u)\subset Z$. To see this, let $z\in Z^{\alpha}$. Thus we have a $c'\in \mathbb{R}^{n+1}$ such that $B(c',r)\subset E(v_{\alpha})$ and $z\in S(c',r)$. So there is a sphere of radius $r$ supporting the graph of $v_{\alpha}$ from above at $z$. If $z\in$ graph($u$), then we must have $z \in Z$:  $B(c',r)\subset E(v_{\alpha})$ and $v_{\alpha}(x)\geq u(x)$ give us that $B(c',r)\cap\text{graph}(u)=\emptyset$ and $c'_{n+1}> u(Pc'),$ which together with $z\in S(c',r)$ imply that $z \in Z$, by definition. Since $U_{\alpha}\subset \text{graph}(u), Z^{\alpha}\cap U_{\alpha}\subset Z^{\alpha} \cap \text{graph}(u)\subset Z$. And of course $Z^{\alpha}\cap U_{\alpha}\subset U_{\alpha}$, so together we have $Z^{\alpha}\cap U_{\alpha} \subset Z\cap U_{\alpha},$ which gives us the second inclusion.

The third inclusion is the critical aforementioned relation between the set of points with spheres of support and a disk in $\mathbb{R}^n$. (Below we will take $\epsilon=R\sin\alpha$ and $\delta=(R-r)\tan\alpha$). To obtain this inclusion we proceed as follows. Let $x\in \mathbb{R}^n$, be such that $|x|<R-r$, and consider the set \begin{align}
\{c'\in \{x\} \times \mathbb{R}: B(c',r)\subset E(v_{\alpha}) \}.
\end{align}
This set is a non-empty, closed half-line. To see this, consider lowering the sphere $S((x,c'_{n+1}),r)$ in $\mathbb{R}^{n+1}$ onto the graph of $v_{\alpha}$, by continuously decreasing the last coordinate. Because the radius of this sphere is r and $|x|<R-r$, this sphere comes down on a closed subset of the graph of $v_{\alpha}$. Once contact is made with the graph of $v_{\alpha}$ we stop, and the corresponding value of $(x,c'_{n+1})$ is our closed endpoint. Let $c'\in\mathbb{R}^{n+1}$ be this endpoint and $y\in S(c',r)\cap \text{graph}(v_{\alpha})$ (note that $y$ may not be unique). Then $c'=\gamma^{\alpha}(y)$ and $x=P\gamma^{\alpha}(y)$, and so
\begin{align}
B_{N}(0,R-r)\subset P\gamma^{\alpha}(Z^{\alpha}).
\end{align}
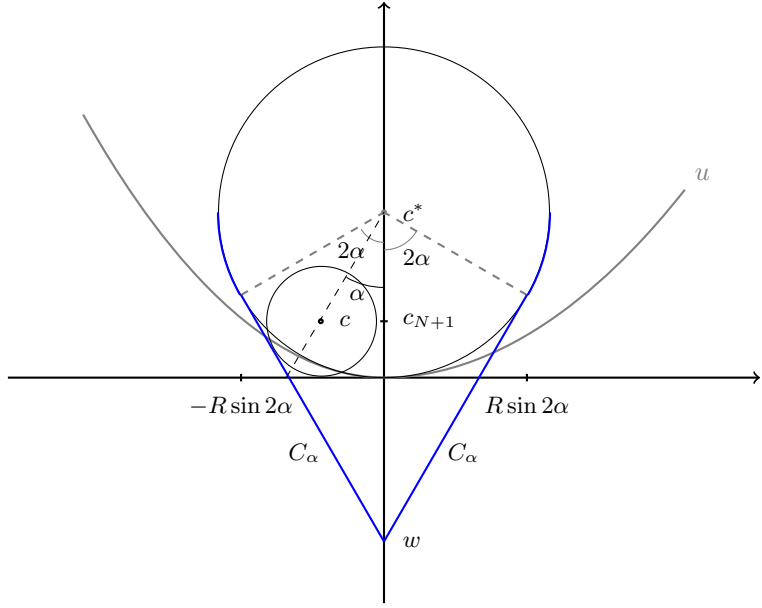
\begin{figure}[H]
\centering
\begin{tikzpicture}

\draw[thick, ->] (-5,0) -- (5,0); %axes
\draw[thick, ->] (0,-3) -- (0,5);

\draw[thick, dashed, gray] (1.9,1.1) -- (0,2.2); %radii
\draw[thick, dashed, gray] (-1.9,1.1) -- (0,2.2);

\draw[thick,gray] (0,0) parabola (4,2.5) node[anchor=south west] {$u$}; % convex function u
\draw[thick,gray] (0,0) parabola (-4,3.5) ;

\draw (0,2.2) circle (2.2cm) ; % supporting sphere 

\draw[thick,gray] (0,2.2) circle (0.025cm) ; %center of supporting sphere
\node at (0,2.2) [label=right:\small{$c^{*}$}]{} ;

\coordinate (center) at (0,2.2); %T_alpha

\cercle{center}{2.2cm}{0}{-30};
\cercle{center}{2.2cm}{180}{30};

\node at  (0,1.7) [label=left:\small{$2\alpha$}]{} ; %alpha
\node at  (0,1.6) [label=right:\small{$2\alpha$}]{} ;

\draw [gray] (0,1.7) arc (270:330:.5cm);
\draw [gray] (0,1.8) arc (270:210:.35cm);

\draw[thick, blue] (0,-2.18) -- (1.9,1.1); %C_alpha
\draw[thick, blue] (0,-2.18) -- (-1.9,1.1);

\node at (0,-2.18) [label=right:\small{$w$}]{} ;

\node at (.6,-1) [label=right:\small{$C_{\alpha}$}]{} ; %label for C_alpha
\node at (-.6,-1) [label=left:\small{$C_{\alpha}$}]{} ;

\draw[thick] (1.9,-.05) -- (1.9,.05);%tick marks
\draw[thick] (-1.9,-.05) -- (-1.9,.05);

\node at (1.9,0) [label=below:\small{$R \sin 2\alpha$}]{} ;
\node at (-1.9,0) [label=below:\small{$-R \sin 2\alpha$}]{} ;

\draw[ dashed] (-1.3,0) -- (0,2.2);
\draw (-.83,.75) circle (.73cm) ;
\draw[thick] (-.84,.75) circle (.025cm) ;
\node at (-.84,.75) [label=right:\small{$c$}]{} ;

\draw (0,1.2) arc (270:240:1cm);

\draw[thick] (-.05,.75) -- (.05,.75);%tick marks

\node at (0,.75) [label=right:\small{$c_{N+1}$}]{} ;

\node at (0,1.1) [label=left:\small{$\alpha$}]{} ; %alpha
\end{tikzpicture}

\caption{Closest Supporting Spheres to Origin}
\end{figure}$\vspace{3mm}$\\

Now $Z^{\alpha}\setminus (C_{\alpha}\cup T_{\alpha}) \subset$ graph$(v_{\alpha})\setminus(C_{\alpha}\cup T_{\alpha})=U_{\alpha}$, so clearly $Z^{\alpha}\setminus (C_{\alpha}\cup T_{\alpha})\subset Z^{\alpha}\cap U_{\alpha}$. Therefore,
\begin{align}
P\gamma^{\alpha}(Z^{\alpha})\setminus P\gamma^{\alpha}(Z^{\alpha}\cap(C_{\alpha}\cup T_{\alpha}))\subset P\gamma^{\alpha}(Z^{\alpha}\cap U_{\alpha}).
\end{align}
This relation and (12) will give us our third inclusion (9), once we show that 
\begin{align}
 P\gamma^{\alpha}(Z^{\alpha}\cap(C_{\alpha}\cup T_{\alpha}))\cap B_{N}(0,\delta)=\emptyset.
\end{align}

Consider the family of all spheres $S(c',r)$ which support $C_{\alpha}\setminus Y$ from above and are contained in the upper half space $y_{n+1}\geq 0$. Clearly the smallest value of $|P(c')|$ is attained when the sphere $S(c',r)$ is tangent to both $C_{\alpha}$ and $\{y_{n+1}=0\}$ (see Fig. 3). It is not difficult to see that in this case $\varangle (c'-c,0-c)=\alpha$, where $c$ here is the center of the initial sphere of support. 
This gives us $|P(c')|=(|c|-c'_{N+1})\tan\alpha=(R-r)\tan\alpha=\delta$, which implies 
\begin{align}
 P\gamma^{\alpha}(Z^{\alpha}\cap C_{\alpha})\cap B_{N}(0,\delta)=\emptyset.
\end{align}
Now when $S(c',r)$ supports $T_{\alpha}\setminus Y$ from the above at some point $y$, the segment $\overline{c',y}$ is normal to $S(c,R)$ and $y_{N+1}\geq R(1-\cos 2\alpha)\geq \delta$. Thus $\varangle (c'-c,0-c)\geq 2\alpha$ and, as above, $|P(c)|\geq (R-r)\tan 2\alpha \geq \delta$ (note $0\leq \alpha\leq\frac{\pi}{4}$). This gives 
\begin{align}
 P\gamma^{\alpha}(Z^{\alpha}\cap T_{\alpha})\cap B_{N}(0,\delta)=\emptyset.
\end{align}
Combining (15) and (16) we have (13), which gives the third inclusion.\\
\textbf{\textit{Step Four.}} Estimate of the density of $X$. The above inclusions and the effect of Lipschitz maps on measure, will be enough to estimate the density of $X=P(Z)$. Recall that $Z=\{(x,u(x)) \in \mathbb{R}^{N+1} | x \in X \}$, where $X$ is the set of points in $B(0,d)\subset \mathbb{R}^{N}$ such that there exists a sphere of radius $r$ supporting  the graph of $u$ from above at $(x,u(x))$.

Using a few theorems from Rockafellar [7], it can be shown that the map $\varphi : P(U_{\alpha})\rightarrow U_{\alpha}$, where $\varphi(x)=(x,u(x))$ is Lipschitz with constant $(1+g_{\alpha}^{2})^{\frac{1}{2}}$, where $g_{\alpha}=\text{sup} \{ |\nabla u| : |x|< R \sin 2\alpha \}$. More specifically, by Theorem 10.4, $u$ is Lipschitz, and by Theorems 24.7, 25.5, and 25.6 $g_{\alpha}$ is a Lipschitz bound for $u|_{B(0,R\sin 2\alpha)}$). A simple Pythagorean argument then shows $(1+g_{\alpha}^{2})^{\frac{1}{2}}$ is a Lipschitz bound for  $\varphi$. Notice that $\varphi$ maps $X\cap P(U_{\alpha})= P(Z\cap U_{\alpha})$ onto $Z\cap U_{\alpha}$.

A basic theorem regarding the effect of Lipschitz maps on Hausdorff measures (Theorem 2.29 in Rogers [8]), along with our first inclusion from above (7), leads to:
\begin{align*}
H^{n}(Z\cap U_{\alpha})\leq & (1+g_{\alpha}^{2})^{\frac{n}{2}}m_{n}(X\cap P(U_{\alpha}))\\
\leq & (1+g_{\alpha}^{2})^{\frac{n}{2}}m_{n}(X\cap B(0,\varepsilon)),\hspace{3mm} \varepsilon = R\sin 2\alpha,
\end{align*} where again $H^n$ and $m_n$ denote the Hausdorff and Lebesgue measure on $\mathbb{R}^n$, respectively.
Furthermore
\begin{align*}
m_{n}(B(0,\delta)) \leq & m_{n}(P\gamma^{\alpha}(Z^{\alpha}\cap U_{\alpha})) \hspace{6mm} \text{by } (9)\\
\leq & H^{n}(Z^{\alpha}\cap U_{\alpha}) \hspace{16mm} P\gamma^{\alpha} \text{ is Lipschitz with constant }\leq 1\\
\leq & H^{n}(Z\cap U_{\alpha}) \hspace{18mm} \text{by } (8).
\end{align*}
Finally, combining these inequalities one obtains
\begin{align*}
\dfrac{m_{n}(X\cap B(0,\varepsilon))}{m_{n}(B(0,\varepsilon))}&\geq(1+g_{\alpha}^{2})^{\frac{-n}{2}}\dfrac{m_{n}( B(0,\delta))}{m_{n}(B(0,\varepsilon))}\\
&=(1+g_{\alpha}^{2})^{\frac{-n}{2}} \left(\dfrac{(R-r)\tan \alpha}{R \sin 2\alpha}\right)^{n}\\
&=(1+g_{\alpha}^{2})^{\frac{-n}{2}} \left(\dfrac{R-r}{2R }\right)^{n} \cos^{-2n}\alpha,
\end{align*}
where the volume of an $n$-ball of radius $r$ is $\dfrac{\pi^{\frac{n}{2}}r^{n}}{\Gamma (\frac{n}{2}+1)}$ in the first equality, and $\Gamma$ denotes the gamma function. Thus,
\begin{align*}
\liminf_{\varepsilon\rightarrow 0}\dfrac{m_{n}(X\cap B(0,\varepsilon))}{m_{n}(B(0,\varepsilon))} \geq \liminf_{\varepsilon\rightarrow 0} (1+g_{\alpha}^{2})^{\frac{-n}{2}} \left(\dfrac{R-r}{2R }\right)^{n} \cos^{-2n}\alpha.
\end{align*}
Now since $\varepsilon = R\sin 2\alpha$ and $0<\alpha<\frac{\pi}{4}$, as $\varepsilon\rightarrow 0$, $\alpha\rightarrow 0$. And as the gradient of a convex function is continuous (Theorem 25.5, [1]), $g_{\alpha}\rightarrow 0$ as well since $\nabla u(0)=0$. Therefore the lower density of $X$ at $0$ is not less than $\left(\dfrac{R-r}{2R }\right)^{N}$. 
\end{proof}

\subsection{2.4 Proof of Theorem 1.3}
Lemma 1.7 and Proposition 1.6 now combine nicely to give us Theorem 1.3.
\begin{proof}[Proof of Theorem 1.3]
First, we prove the density result. Without loss of generality, let $x_0=0, u(x_0)=0, \nabla u(x_0)=0$. Note that by the convexity of $u$ this implies $u\geq 0$. Set $k_0=K(u,x_0)=K(u,0)$, and let $k>k_0$ be fixed and take $K$ such that $k>K>k_0$.

Set $R=\dfrac{1}{K}$ and note that $R -(R^{2}-|x|)^{\frac{1}{2}}\geq \dfrac{1}{2R} |x|^{2}=\dfrac{K}{2} |x|^{2}, \hspace{2mm} \forall x$ such that $|x|<R$. This follows immediately by contradiction. The left-hand side of this inequality is the  last component of the point $(x,t)\in \mathbb{R}^{n}$, where $x\in \mathbb{R}^{n}$, on the $(n+1)$-dimensional sphere of radius $R$ centered $(0,...,0,R) \in \mathbb{R}^{n+1}$ (i.e the value of $d(x)$, where $d$ is the lower hemisphere function defined in the proof of the proposition, see Fig. 1).
\begin{figure}[!h]
\centering
\begin{tikzpicture}

\draw[thick, ->] (-2.5,0) -- (2.5,0); %axes
\draw[thick, ->] (0,-1) -- (0,4.5);

\draw (0,0) parabola (2,.75);

\draw (0,2) circle (2cm) ; % supporting sphere 

\node at (0,2) [label=right:\small{$R$}]{} ;
\node at (2,.75) [label=right:\small{$\dfrac{1}{2R} |x|^{2}$}]{} ;

\draw[thick] (1,.1) -- (1,-.1);%tick marks
\draw[thick] (-.1,2) -- (.1,2);%tick marks

\node at (1,0) [label=below:\small{$x$}]{} ;
\node at (.2,1.2) [label=left:\tiny{$\left( R^{2} - |x|^{2} \right)^{\frac{1}{2}}$}]{} ;

\node at (2.5,0) [label=below:\small{$\mathbb{R}^n$}]{} ;
\node at (0,4.5) [label=left:\small{$\mathbb{R}$}]{} ;

\draw[dashed] (0,2) -- (1,.28);
\draw[dashed] (0,.28) -- (1,.28);

\end{tikzpicture}
\caption{}

\end{figure}

Since $K>K(u,0)$, by Proposition 1.6 (ii) there exists $d>0$ such that 
\begin{align*}
u(0+h)-u(0)- \langle\nabla u(0),h\rangle \hspace{2mm} &\leq \frac{1}{2}K|h|^{2} \text{ for every } |h|<d.
\end{align*} So
\begin{align*}
u(h)& \leq \frac{1}{2}K|h|^{2} \text{ for every } |h|<d.
\end{align*}
Thus the sphere $S(c,R)$, where $c=(0,...0,R)\in \mathbb{R}^{n+1}$, supports the graph of $u|_{B(0,d)}$ form above at $0\in \mathbb{R}^{n+1}$, and Lemma 1.7 can be applied to the function $u|_{B(0,d)}$.

Let $r$, such that $\frac{1}{k}<r<R$, be arbitrary, and let $X=X_{r}$ and $Z=Z_{r}$ be defined as in Lemma 1.7. By Proposition 1.6 (iii), $\forall x \in X$
\begin{align*}
K(u,x)\leq \dfrac{(1+g^{2})^{\frac{3}{2}}}{r}, \hspace{5mm} \text{ where }g=|\nabla u(x)|.
\end{align*}
Set 
\begin{align*}
g_{\varepsilon}=\sup\{|\nabla u(x)| : |x|<\varepsilon\}.
\end{align*} Then clearly $$K(u,x) \leq \dfrac{(1+ g_{\varepsilon}^{2})^{\frac{3}{2}}}{r} \hspace{5mm} \forall x\in X\cap B(0,\varepsilon).$$

By the continuity of the gradient function, $\lim_{\varepsilon\rightarrow 0} g_{\varepsilon}=|\nabla u(0)|=0$. Thus since $\dfrac{1}{r} < k$, there exists $\varepsilon'$, where $0<\varepsilon'< d$, such that 
\begin{align*}
\dfrac{(1+ g_{\varepsilon}^{2})^{\frac{3}{2}}}{r}<k, \hspace{5mm} \text{ for } 0<\varepsilon<\varepsilon', 
\end{align*}
and so
\begin{align*}
(B(0,\varepsilon)\cap X)\subset (B(0,\varepsilon)\cap X'_{k}), \hspace{5mm} \text{for } 0<\varepsilon<\varepsilon'.
\end{align*}
If $x\in X$ then there exists a supporting sphere of radius $r$ at $(x,u(x))$, and if $x \in B(0,\varepsilon)$, where $\varepsilon <\varepsilon'$, then $K(u,x)<k$.

It follows by Lemma 1.7 that
\begin{align*}
\liminf_{\varepsilon\rightarrow 0}\dfrac{m_{n}(X'_{k}\cap B(0,\varepsilon))}{m_{n}(B(0,\varepsilon))} \geq & \liminf_{\varepsilon\rightarrow 0}\dfrac{m_{n}(X\cap B(0,\varepsilon))}{m_{n}(B(0,\varepsilon))}\\
\geq & \left( \dfrac{R-r}{2R} \right)^n.
\end{align*}
Now recall that $R=\dfrac{1}{K}$ was chosen arbitrarily so that it satisfied the inequality $\dfrac{1}{k} <\dfrac{1}{K}<\dfrac{1}{k_0}$, where $k$ and $k_0$ are fixed. Similarly, $r$ was chosen arbitrarily so that $\dfrac{1}{k} < r <\dfrac{1}{K}$. Thus we can choose $R=\dfrac{1}{K}$ and $r$ arbitrarily close to $\dfrac{1}{k_0}$ and $\dfrac{1}{k}$, respectively, giving us the desired bound $\left(\dfrac{k-k_0}{2k}\right)^{n}.$ \\

Finally, the fact that $X'_{k}:=\{ x\in \text{dom } (u) : K(u,x)<k  \}$ is Borel is contained in Proposition 2.2 and Lemma 2.3 below. Let $u: \mathbb{R}^n\rightarrow \mathbb{R}$ convex. Proposition 2.2 shows that the set $W$ on which $u$ is differentiable is Borel, specifically a $F_{\sigma \delta}$, and Lemma 2.3 proves that $K(x):=K(u,x)$ is of second Baire class on this set. Since $K(x)= +\infty$ where $\nabla u$ doesn't exists, $$X'_{k}=\{ x\in \text{dom } (u) : K(x)<k  \} =\{ x\in W : K|_W(x)<k  \}.$$ It follows immediately that $X'_k$ is Borel, as $K|_W : W \rightarrow \mathbb{R}$ is a Borel measurable function. Recall that Baire class 1 functions are the pointwise limit of continuous functions and thus Borel measurable, and Baire class 2 functions are the pointwise limit of Baire class 1 functions and thus also Borel measurable.
\end{proof}
\begin{proposition}{2.2}
Let $u:\mathbb{R}^n\rightarrow \mathbb{R}$ be convex. Then the set on which $u$ is differentiable is a dense Borel set, specifically an $F_{\sigma \delta}$.
\end{proposition}
\begin{proof}
Since $u$ is convex, $u$ is differentiable at $x$ if and only if all the partial derivatives of $u$ exist at $x$, with respect to any basis [6, IV.4.2]. Let $\{e_i\}_{i=1}^{n}$ be the standard basis in $\mathbb{R}^n$, and define 
$$f'(x,e_i):= \lim_{t\downarrow 0} \dfrac{f(x+te_i)- f(x)}{t}.$$
Then $\frac{\partial f}{\partial x_i}(x)$ exists if and only if $f'(x,e_i)=-f'(x,-e_i)$[6, IV.4.2]. Note that the above limit always exists for a convex function and $f'(x,e_i)\geq -f'(x,-e_i)$ for all $x$. Take $E$ to be the set where $u$ is not differentiable and $E_i$ to be the set of points where $\frac{\partial f}{\partial x_i}(x)$ does not exist. Then $E=\cup_{i=1}^n E_i$, and 
$$E_i = \{ f'(x,e_i)+f'(x,-e_i)>0  \}.$$
If $x\in E_i$, then there exists $N$ such that for all $n\geq N$, 
$$\dfrac{f(x+\frac{e_i}{k})- f(x)}{\frac{1}{k}} + \dfrac{f(x-\frac{e_i}{k})- f(x)}{\frac{1}{k}} > \dfrac{1}{n},$$
for all $k\geq n$.
Let 
$$ E_{n,k}=\left\{ x: \frac{f(x+\frac{e_i}{k})- f(x)}{\frac{1}{k}} + \frac{f(x-\frac{e_i}{k})- f(x)}{\frac{1}{k}} > \frac{1}{n}   \right\},$$ and note that $E_{n,k}$ is open since $f$ is continuous (a real-valued convex function).
Thus, $$E_i = \cup_{n=1}^\infty \cap _{k=n}^\infty E_{n,k},$$ which is clearly a $G_{\delta \sigma}$ and so $E$ is also a $G_{\delta \sigma}$, being a union of finitely many. Therefore, the set $\mathbb{R}^n \setminus E$ on which $u$ is differentiable  is an $F_{\sigma \delta}.$ That $\mathbb{R}^n \setminus E$ is dense is well-known.
\end{proof}

\begin{lemma}{2.3}
Let $u:\mathbb{R}^n\rightarrow \mathbb{R}$ be convex and $W \subset \mathbb{R}^n$ the set on which $u$ is differentiable. Then the function $K(x):=K(u,x)$ is of second Baire class on $W$.
\end{lemma}
\begin{proof}
We follow notes of Slodkowski, not contained in [9], for this proof. Let
$$f(x,\epsilon)= \frac{2}{\epsilon^2} \max \left\{ u(x+\epsilon h) - \epsilon \langle \nabla u(x), h \rangle : |h|=1    \right\}.$$
Then $K(x)= \limsup_{\epsilon >0} f(x, \epsilon).$ Since $u$ is convex, $\nabla u(x)$ is continuous on $W$, and so $f(x, \epsilon)$ is a continuous function on $W \times (0, \infty)$.

Next, let $$g(x,n)= \sup \left\{ f(x, \epsilon): 0< \epsilon< \frac{1}{n}  \right\}.$$
Since $g(\cdotp,n)$ is the supremum of a family of continuous functions it is lower semicontinuous, and thus the limit of an increasing sequence of continuous functions on $W$. Therefore,$g(\cdotp,n)$ is of first Baire class.

Now, note that $$\limsup_{\epsilon>0} f(x, \epsilon)= \lim_{n\rightarrow \infty}g(x,n),$$
and thus $K(x)$ is of second Baire class as it is the limit of Baire class one functions.
\end{proof}

\section{3. Dual Perspective}
\subsection{3.1 Background} 
Since $u$ is convex near $x_0$, it is natural to study this quantity $K(u,x_0)$ from the dual perspective as well. Let $Cvx(\mathbb{R}^n)$ denote the space of convex, lower semi-continuous functions on $\mathbb{R}^n$. Given a function $u\in Cvx(\mathbb{R}^n)$, one can apply  the Legendre--Fenchel transform $\mathscr{L}:Cvx(\mathbb{R}^n)\rightarrow Cvx(\mathbb{R}^n)$ of $u$ to obtain its conjugate or dual function $u^*$, where 
$$ u^*\equiv \mathscr{L}u(s)=\sup_{x} (\langle s,x\rangle-u(x)).$$
$\mathscr{L}$ is an order-reversing, involutive transform on $Cvx(\mathbb{R}^n)$, and for sufficiently nice convex functions (differentiable, strictly convex, and 1-coercive), $u^*$ is given by
$$u^*(s)=\langle s, (\nabla u)^{-1}(s) \rangle - u((\nabla u)^{-1}(s)).$$
The conjugate function $u^*$ can be viewed as a reparametrization of the original function $u$ in terms of its tangents using the duality between points and hyperplanes. More specifically, given a vector in $\mathbb{R}^n$, there is an associated family of hyperplanes with that gradient. $u^*$ distinguishes the one that supports the epigraph of $u$ by specifying a point on that plane.

For convex functions defined only in a neighbourhood it is standard to extend the function to all of $\mathbb{R}^n$ by setting it equal $+\infty$ outside that neighbourhood. In our case, we are given $u$ convex near $x_0$, so we extend it in this manner, if necessary. Clearly this does not affect $K(u,x_0)$, which is a purely local property. 
Recall the following basic definitions:
\begin{definition}{3.1}
\textit{The differentiable function $f:\mathbb{R}^n\rightarrow \mathbb{R}$ is convex if for all $x,x'\in \mathbb{R}^n$ $$f(x')\geq f(x)+ \langle \nabla f(x),(x'-x) \rangle,$$ and strictly convex if the inequality is strict for $x\neq x'$.}
\end{definition}
\begin{definition}{3.2}
\textit{The differentiable function $f:\mathbb{R}^n\rightarrow\mathbb{R}$ is strongly convex with modulus $c$ if and only if for all $(x,x')\in \mathbb{R}^n \times \mathbb{R}^n$,} 
\begin{align*}
f(x')\geq f(x)+ \langle\nabla f(x), (x' - x)\rangle +\frac{1}{2} c |x' - x|^2.
\end{align*}
\end{definition}
When $f$ is not differentiable a lot of analysis can still be done using the calculus of subdifferentials.
\begin{definition}{3.3}
\textit{Let $f:\mathbb{R}^n\rightarrow\mathbb{R}$ be convex. The subdifferential of $f$, denoted $\partial f$, is a set function, where $\partial f(x)= \left\{ s\in \mathbb{R}^n: f(y)\geq f(x)+\langle s,y-x\rangle \text{  }  \forall y\in \mathbb{R}^n \right\}.$}
\end{definition}
Under the Legendre transform, differentiability of $u$ corresponds to convexity or monotonicity of $u^*$. Recall from Proposition 1.8, two properties that transform especially well are (i) $u\in C^1$ if and only if $u^*$ is strictly convex, and (ii) $u\in C^{1,1}$, where $\nabla u$ has Lipschitz constant $c$ if and only if $u$ is strongly convex with modulus $\frac{1}{c}$. 

\subsection{3.2 Quadratic convexity}
In this section we look at how a bound on $K(u,x_0)$ or equivalently a sphere of support to the graph of $u$ at $(x_0,u(x_0))$ transforms to a property of $u^*$. More specifically, since $K$ or a sphere of support is a bound on a generalized second-order derivative of $u$, how does this translate to information about the convexity of $u^*$? We should expect a more localized property then in Proposition 1.8, as we only have information at $x_0$. Further, we are not assuming any regularity beyond differentiable at $x_0$.

Now, strong convexity may also defined in terms of quadratic functions: $u$ is strongly convex with modulus $m$ if $u-\frac{1}{2}m|x|^2$ is convex. Similarly, quasi-convexity, is defined via quadratics: $u$ is $\lambda$- quasi-convex if $u+ \frac{1}{2}\lambda |x|^2$ is convex.

Let $u:\mathbb{R}^n \rightarrow \mathbb{R}$ be convex with $K(u,x_0)=k_0 <\infty$. By the definition of $K(u,x)$, for any $k>k_0$ there exists $\epsilon >0$ such that
$$u(x_0+h)-u(x_0)-\langle \nabla u(x_0), h \rangle \leq \frac{1}{2} k|h|^2, \text{ for all } |h|< \epsilon .$$
This motivates the following definition.
\begin{definition}{3.4}
\textit{Let $f:\mathbb{R}^n \rightarrow \mathbb{R}$ be convex. Then $f$ is quadratically (resp. sub-quadratically) convex at $x_0$ with modulus $m>0$ if there exists $\epsilon >0$ and a quadratic function $Q:\mathbb{R}^n\rightarrow \mathbb{R}$ with $\nabla^2 Q=mI$ such that 
$$f(x_0)=Q(x_0) \text{ and } f(x) \geq Q(x),\hspace{6.6mm} \forall x\in B(x_0,\epsilon)$$ resp. 
$$f(x_0)=Q(x_0) \text{ and } f(x) \leq Q(x),\hspace{6.6mm} \forall x\in B(x_0,\epsilon).$$ }
\end{definition}

\begin{example}{3.5}
$f(x)=|x|^{4/3}$ is quadratically convex at 0, but not sub-quadratically convex at 0. Note also that $K(f,0)=+\infty$ and it does not have a sphere of support at 0.
\end{example}

\begin{example}{3.6}
More generally, consider any function of the form $f(x)=A|x|^k$, at $x=0$. If $0<k<1$, $f$ is not convex. If $k=1$, $f$ is quadratically convex at 0, but not sub-quadratically convex. If $1<k<2$ then $f$ is strictly convex and quadratically convex but not sub-quadratically convex. If $k=2$, $f$ is both quadratically convex and sub-quadratically convex. If $k>2$, $f$ is sub-quadratically convex but not quadratically convex. 
\end{example}
If $f$ is of the form $f=\frac{|x|^k}{k}$, then $f^*=\frac{|y|^q}{q}$, where $\frac{1}{k}+\frac{1}{q}=1$. So, in general, given that the Legendre-Fenchel transform is order-reversing and quadratics are transformed into quadratics, it follows that if $f$ is quadratically convex, $f^*$ is sub-quadratically convex. For a convex $C^2$ function $f$, if $\nabla^2f(x_0)$ is positive definite then $f$ is both quadratically and sub-quadratically convex at $x_0$.
\begin{proof}[Proof of Theorem 1.9]
Suppose $K(u,x_0)=k_0<\infty$. As stated above, by definition of $K(u,x_0)$, for any $k>k_0$, there exists $\epsilon >0$ such that $u$ satisfies 
$$u(x)-u(x_0)-  \langle \nabla u(x_0), x-x_0 \rangle \leq \frac{1}{2} k|x-x_0|^2,$$ for all $x\in B(x_0,\epsilon)$. Thus, on this neighbourhood of $x_0$
$$u(x) \leq u(x_0)+  \langle \nabla u(x_0), x-x_0 \rangle +\frac{1}{2} k|x-x_0|^2.$$
By assumption $u$ is convex, and $k>k_0\geq 0$, so the right-hand side is also convex. Taking the Legendre transform gives
$$ u^*(y)\geq \langle \nabla u(x_0),x_0 \rangle - u(x_0) +\langle x_0, y- \nabla u(x_0)\rangle+ \frac{1}{2}k\left|\dfrac{y-\nabla u(x_0)}{k}\right|^2.$$
Now $u^*$ may not be differentiable at $\nabla u(x_0)$, however $\nabla u(x_0) \in \partial u(x_0)$ if and only if $x_0\in \partial u^*(\nabla u(x_0))$, which is equivalent to $u^*(\nabla u(x_0))=\langle \nabla u(x_0), x_0 \rangle -u(x_0)$. So the above inequality simplifies to
$$u^*(y)\geq u^*(\nabla u(x_0))+ \langle x_0, y- \nabla u(x_0) \rangle + \frac{1}{2k}|y-\nabla u(x_0)|^2.$$
Note that there is equality at $y_0=\nabla u(x_0)$ and the Hessian of the right-hand side is $\frac{1}{k}I$ so $u^*$ is quadratically convex with modulus $\frac{1}{k}$.

On the other hand, if $u^*$ is quadratically convex at $y_0=\nabla u(x_0)$ with modulus $\frac{1}{k}$ then $u$ will be sub-quadratically convex with modulus $k$ at $x_0$, and it follows that $K(u,x_0)\leq k$.
\end{proof}
In the above proof we do not need to worry about $\partial u(B(x_0,\epsilon))$ being degenerate (for example if $u$ is locally a hyperplane at $x_0$) because in that case $u^*(y)$ will then be $+\infty$ away from $\nabla u(x_0)$ so clearly the inequality will hold on some neighbourhood.

Our goal now is to obtain the nice bound on $K(u,x)$ in Proposition 1.6 using the dual function, given a sphere of support to the graph of $u$ at $(x,u(x))$. The following elementary lemma, which we state without proof, will help us to reduce arguments on $\mathbb{R}^n$ to ones on $\mathbb{R}$.
\begin{lemma}{3.7}
\textit{Let $S_{r}$ be an $n$-sphere with radius $r$ in $\mathbb{R}^{n+1}$, centered at $(0,...,0,r)$, and let $d:\mathbb{R}^{n}\rightarrow \mathbb{R}$ be the function defined by the lower hemisphere, i.e., for $z\in B_n (0,r)$, $d(z)= r- \sqrt{r^{2}-|z|^{2}}$. Then for any $x\in B_n(0,r)$ and $v \in \mathbb{R}^n$, $|v|=1$, the graph of $\psi:I \subset \mathbb{R}\rightarrow \mathbb{R}^{n+1}$ defined by $\psi(t)=d(x+tv)$ is a lower semi-circle in $\mathbb{R}^{n+1}$ of radius $\leq r$, where $I=(-\epsilon, \epsilon')$ is of maximal length.}
\end{lemma}

\begin{proposition} {3.8}
\textit{Let $f:\mathbb{R}^n\rightarrow \mathbb{R}$ be $C^2$ and convex and suppose there exists a sphere of support to the graph of $f$ at $(x_0,f(x_0))$ of radius $r$. Then} $$K(f,x_0) \leq \dfrac{(1+\nabla f|_{x_0}^2)^\frac{3}{2}}{r}.$$
\end{proposition}
\begin{proof}
Because $f$ is $C^2$, $K(f,x_0)$ is the largest eigenvalue $\lambda_{max}$ of $\nabla^2 f(x_0)$. If $\lambda_{max}$=0 or $\nabla f(x)=0$ then the bound on $K(f,x_0)$ is trivial, so let $\lambda_{max}>0$ and $\nabla f(x)\neq 0$. $f$ is convex so $\nabla^2 f(x_0)$ is symmetric positive semi-definite, and there exists an orthonormal basis of eigenvectors. Let $v$ be the eigenvector coresponding to $\lambda_{max}$. By duality, $v$ is also an eigenvector corresponding to $\lambda^*_{min}= \frac{1}{\lambda_{max}}$, the smallest eigenvalue of $\nabla^2 f^*(\nabla f(x_0))$. This follows from the fact that the Hessians of dual functions satisfy
$$
\nabla^2 f^*(y_0) = \nabla^2 f(x_0)^{-1}, \hspace{10mm} \text{ where }y_0 = \nabla f(x_0).
$$ (Here we assume without loss of generality that $\nabla^2 f(x_0)$ is invertible because we are only concerned with $\lambda_{max}>0$).\\
Let $S((c,t),r)$ be the sphere of support of radius $r$, to the graph of $f$ at $x_0$, and $d$ the associated lower hemisphere function, i.e.
\begin{align*}
d(x) = & t - \sqrt{r^2 -|x-c|^2},\hspace{2mm} x\in \bar{B}(c,r)\\
d(x) =  &\infty, \text{ else}.
\end{align*}
Clearly $d$ is convex and $d \geq f$, by definition of a supporting sphere. Also, recall that $f$ and $d$ agree up to first order at $x_0$. 

Again by basic properties of the Legendre transform, the following relations hold:
\begin{align*}
f^*(y_0)=d^*(y_0) \hspace{7mm} f^* \geq d^* \hspace{7mm} \nabla f^*(y_0)= \nabla d^*(y_0)= x_0.
\end{align*}
It follows that  $$\lambda^*_{min} \geq \gamma^*_{min}$$ where $\gamma^*_{min}$ is the smallest eigenvalues of $\nabla ^2 d^*(\nabla f(x_0))$. Note that this is equivalent to 
$$\lambda_{max}\leq \frac{1}{\gamma^*_{min}}.$$
Given this bound, we now show that $\gamma^*_{min}$ can always be computed using a function on $\mathbb{R}$.

Let $v'$ be the unit-length eigenvector corresponding to $\gamma^*_{min}$ and $\gamma_{max}$. By Proposition 2.1, $v'$ is in the direction of $\nabla d(x_0).$ By Lemma 3.7, $\tilde {d}$, the restriction of $d$ to this $1-$dimensional subspace defines a lower semi-circle function, and this function has the properties: $\tilde{d}'(x_0)= \langle \nabla d(x_0),v' \rangle = |\nabla d(x_0)|$ and $\tilde{d}''(x_0) = \gamma_{max}.$ Therefore, the dual function $\tilde{d}^*$ has second derivative at $|\nabla f(x_0)|$ equal to $ \gamma^*_{min}$, and so we may assume without loss of generality that $f$ and $d$ are functions on $\mathbb{R}$.

Now we compute $d^*$ directly by using the Legendre transforms of common functions. Rewriting $d$
\begin{align*}
d(x) = & t - \sqrt{r^2 -(x-c)^2}\\
=& t - r \sqrt{1 -\left(\frac{x}{r}-\frac{c}{r}\right)^2},
\end{align*}
and then applying the following well-known conjugate pairs:
\begin{align*}
h(x)= - \sqrt{1-x^2} &\hspace{.5 in} h^*(y)=\sqrt{1+y^2}\\
g(x)= \alpha + \beta x + \gamma u(\lambda x + \delta) & \hspace{.5 in} g^*(x)= -\alpha -\delta \dfrac{y-\beta}{\lambda} + \gamma u^*(\dfrac{y- \beta}{\gamma \lambda}),
\end{align*}
gives
\begin{align*}
d^{*}(y)=& -t +cy + r \sqrt{1+y^2}\\
\frac{d}{dy}d^{*}(y)=& c + \dfrac{ry}{\sqrt{1+y^2}}\\
\frac{d^2}{dy^2} d^{*}(y)=& \dfrac{r}{(1+y^2)^\frac{3}{2}}.
\end{align*}

Thus,
\begin{align*}
K(f,x_0)=\lambda_{max} \leq & \dfrac{1}{ \frac{d^2}{dy^2} d^{*}(|\nabla f(x_0)|)}= \dfrac{(1+|\nabla f(x_0)|^2)^\frac{3}{2}}{r}.
\end{align*}
\end{proof}
The more general case, where $f$ is not assumed to be $C^2$, will use Proposition 3.8 and quadratic convexity of the dual.
\begin{proposition}{3.9}
\textit{Let $f:\mathbb{R}^n \rightarrow \mathbb{R}$ be convex with a sphere of support at $x_0$ of radius $r$. Then $K(f,x_0) \leq \dfrac{(1+\nabla f|_{x_0}^2)^\frac{3}{2}}{r}$.}
\end{proposition}
\begin{proof}
Let $d$ be the lower hemisphere function. Then $d(x_0)=f(x_0)$, and $$ d\geq f \Rightarrow f^*\geq d^*.$$ If $y_0=\nabla f(x_0)$ (which exists since there is a sphere of support) then $$d^*(y_0)=f^*(y_0) \text{ and } \nabla d^*(y_0)\in \partial f^*(y_0).$$ From Proposition 3.8 the smallest eigenvalue of $\nabla ^2 d^*(y_0)$ is equal to $\frac{r}{(1+|y_0|^2)^{\frac{3}{2}}}$, so for any $m < \frac{r}{(1+|y_0|^2)^{\frac{3}{2}}}$ there exists a neighbourhood $U$ of $x_0$ such that $$ f^*(y)\geq d^*(y)\geq d^*(y_0) + \langle \nabla d^*(y_0), y-y_0\rangle + \frac{1}{2} m|y-y_0|^2.$$ Thus, $f^*$ is quadratically convex with modulus $m$. 

It follows that $f= (f^*)^*$ is sub-quadratically convex at $x_0$ with modulus $\frac{1}{m}.$ Let $Q_m$ be a satisfying quadratic. This implies that 
$$K(f,x_0)\leq K(Q_m,x_0)=\frac{1}{m},$$ and since this holds for any $m< \frac{r}{(1+|y_0|^2)^{\frac{3}{2}}}$, $$K(f,x_0)\leq\frac{(1+|y_0|^2)^{\frac{3}{2}}}{r}=\frac{(1+|\nabla f(x_0)|^2)^{\frac{3}{2}}}{r}.$$
\end{proof}

\section{Appendix}
\subsection{A.1 Lipschitz gradient}
Here we show that the generalized derivative $K(f,x)$ retains the following standard property regarding the derivative of a Lipschitz continuous function.
\begin{proposition}{A.1}
Suppose $f:\mathbb{R}^n\rightarrow \mathbb{R}$ is convex and $C^{1,1}$ (i.e $f$ is differentiable and has Lipschitz gradient), with Lipschitz constant $L$. Then $K(f,x)\leq L$ for all $x$. 
\end{proposition}
\begin{proof}
Let $x_0\in \mathbb{R}^n$. $$K(f,x_0) := \limsup_{\epsilon \rightarrow 0} 2\epsilon^{-2} \text{ max } \{ f(x_0 + \epsilon h)  - f(x_0) - \epsilon \langle \nabla f(x_0), h \rangle : |h| =1 \},$$
which can be can written as
$$K(f,x_0)=  \limsup_{\epsilon \rightarrow 0} \text{ max } \left\{ 2 \dfrac{f(x_0 + \epsilon h)  - f(x_0) - \epsilon \langle \nabla f(x_0), h \rangle }{\epsilon^2} : |h|= 1 \right\}.$$
Differentiability lets us use the Cauchy mean value theorem. Let $\phi_1(\epsilon)= f(x_0+\epsilon h) -\epsilon \langle \nabla f(x_0),h\rangle$, and $\phi_2(\epsilon)=\epsilon^2$. Note that 
$$2 \dfrac{f(x_0 + \epsilon h)  - f(x_0) - \epsilon \langle \nabla f(x_0), h \rangle }{\epsilon^2}= 2 \dfrac{\phi_1(\epsilon)- \phi_1(0)}{\phi_2(\epsilon)- \phi_2(0)}.$$
Thus, there exists $\gamma\in (0,\epsilon)$ such that 
\begin{align*}
2 \dfrac{\phi_1(\epsilon)- \phi_1(0)}{\phi_2(\epsilon)- \phi_2(0)} =2 \dfrac{\phi_1'(\gamma)}{\phi_2'(\gamma)} = & \dfrac{\langle\nabla f(x_0+\gamma h),h\rangle - \langle\nabla f(x_0),h\rangle}{\gamma}\\
= & \dfrac{\langle\nabla f(x_0+\gamma h)- \nabla f(x_0) ,h\rangle }{\gamma}\\
\leq & \dfrac{|\nabla f(x_0+\gamma h)- \nabla f(x_0)| }{\gamma} \leq L
\end{align*}
Therefore $K(f,x_0)\leq L$, and thus $\frac{1}{K(f,x_0)}$ bounds the modulus of convexity of $f^*$, for any $x_0$.
\end{proof}
\subsection{A.2 Example of a non $C^{1,1}$ function with a sphere of support}
\begin{example}{A.2} It may seem that since a bound on $K(u,x)$ implies a sphere of support to the graph of $u$ at $(x,u(x))$, that this in turn implies some kind Lipschitz continuity of the gradient in a small neighbourhood of $x$. Here we construct an example of a strictly convex function $f$ that is $C^1$ and twice differentiable with $K(f,0)<\infty$, but with gradient not Lipschitz in any neighbourhood of 0, to show this is not the case.
Let $f:[-1,1]\rightarrow \mathbb{R}$ be given by $f(0)=0$, and for $x\geq 0$
\begin{align*}
f'(x)=\int_0^x \gamma(t)dt, \hspace{5mm} \text{where } \gamma(t):=n+4 \text{ on } I_n \text{ and } 0 \text{ otherwise,}
\end{align*}
with $I_n= \dfrac{1}{(n+4)^2}[1-\dfrac{1}{(n+4)^2},\hspace{2mm}1]$. Define $f'(-x):=-f'(x)$. 

Then $f'$ is clearly increasing and so $f$ is convex. And for $x_n =\dfrac{1}{(n+4)^2}$,
$$f'(x_n)=\int_{0}^{x_1}\gamma(t)\,dt=\sum_{k\geq n}\frac{1}{(k+4)^3}\leq\int_{n+3}^{\infty}\frac{dt}{t^3}= \frac{1}{2(n+3)^2}<\frac{1}{(n+4)^2}=x_n.$$\\
So we have $f'(x)\leq x$ for all $x \in [0,1]$ and $f'(x)\geq x$ for all $x \in [-1,0]$. Since $d'(x)\geq x$ for all $x \in [0,1]$ and $d'(x)\leq x$ for all $x \in [-1,0]$, it follows that the graph of $d$, and thus the unit circle centered at $(0,1)$, is always at or above the graph of $f$, with $f(0)=d(0)$. Therefore, $f$ has a sphere of support at $x_0=0$.

However, there exist sequences $\{x_i\}, \{x_j\}$ such that
$$\dfrac{f'(x_i)-f'(x_j)}{x_i-x_j}$$ blows up: Taking $x_i$ and $x_j$ as the endpoints of $I_n$,
$$ \dfrac{f'(x_i)-f'(x_j)}{x_i-x_j} = \frac{1}{x_i-x_j}\left( \int_0^{x_i} \gamma(t)dt - \int_0^{x_j} \gamma(t)dt\right) = (n+4)^4 \int_{x_j}^{x_i} n+4 dt=n+4.$$
We can make $f$ strictly convex by adding an $x^m$ term, which does not affect any of the above analysis. The above example can be adjusted to show that $f'$ is not $\alpha$-Holder continuous for any $\alpha$.
\end{example}
\subsection{A.3 Osculating and locally supporting spheres }
Here we extend the concept of an osculating circle to a plane curve to that of an \textquotedblleft osculating sphere\textquotedblright to the graph of a function in higher dimensions. The bound on the \textquotedblleft largest eigenvalue \textquotedblright $K(u,x)$ can be seen as a generalization of the relationship between the second derivative of a $C^2$ plane curve $u$ and the radius of its osculating circle:

Let $u:\mathbb{R}\rightarrow \mathbb{R}$ be $C^{2}$. Provided $u''\neq 0$, the radius of curvature at $x$ is defined as
$$r_{u,x}:= \dfrac{1}{\kappa}= \dfrac{(1 + u'^{2})^{\frac{3}{2}}}{u''},$$
where $\kappa$ is the curvature of $u$ at $x$, and the right-hand side is the standard formula for computing the curvature of a planar curve [2, \S 8]. Thus, $$u''=\frac{(1+ u'^2)^{3/2}}{r}.$$

\begin{definition}{A.3}
\textit{The osculating circle, or circle of curvature, to a planar curve $C$ at $p$ is the circle that touches $C$ (on the concave side) at $p$ and whose radius is the radius of curvature of $C$ at $p$.} 
\end{definition}
We extend this to the graphs of $C^2$ convex functions in higher dimensions by
\begin{definition}{A.4}
\textit{For a convex function $u:\mathbb{R}^{n}\rightarrow \mathbb{R}$ let the osculating sphere to the graph of $u$ at $x$ be the $n-$sphere tangent to the graph of $u$ at $x$ the with radius equal to that of $\frac{1}{\lambda_{max}}$.}
\end{definition}
It is easy to show that any tangent sphere at $(x,u(x))$ with radius less than the osculating sphere at that point is a (local) sphere of support. And any tangent sphere at $(x,u(x))$ with radius greater than the osculating sphere cannot be a (local) sphere of support.
\subsection{A.4 Spheres of support to a function and its dual}
Given a convex function $u$ with a sphere of support at $(x_0, u(x_0))$, the conjugate function $u^*$ will not necessarily have a sphere of support at the corresponding point $(\nabla u(x_0), u^*(\nabla u(x_0))$. For example take $u=\frac{1}{4}|x|^4$ and $u^*=\frac{3}{4}|x|^{\frac{4}{3}}$. However, for more regular and sufficiently convex functions (e.g. $C^2$ and locally strongly convex), we will have a sphere of support (locally) to both graphs at corresponding points, and the order-reversing property of $\mathscr{L}$ provide a simple inequality relating the radii of these spheres. We state this without proof.

\begin{proposition} {A.5}
\textit{Let $u:\mathbb{R}^n \rightarrow \mathbb{R}$ be strongly convex and $C^2$ near $x_0$, and suppose $u$ has a sphere of support of radius $r_{x_0}$. If $r_{y_0}$ is the radius of a sphere of support to $u^*$ at $y_0=\nabla u(x_0)$, then
$$
r_{y_0}\leq \dfrac{\left(1+|x|^{2}\right)^{\frac{3}{2}} \left(1+|\nabla u(x_0)|^{2} \right)^{\frac{3}{2}}}{r_{x_0}}.
$$}
\end{proposition}
\section{Acknowledgements}
I am very grateful to Y.A. Rubinstein for introducing me to the work of Slodkowski, Harvey and Lawson, and for his ongoing guidance and encouragement. I thank T. Darvas and R. Hunter for helpful comments and discussions. I would also like to thank Z. Slodkowski for an insightful correspondence, and a referee for their careful review and stellar suggestions.

\section{\normalsize References}
\small 
$[1]$A.D. Alexandrov, \textit{Almost everywhere existence of the second differential of a convex function and properties of convex surfaces connected with it (in Russian)}, Lenningrad State Univ. Ann. Math. 37 (1939), 3-35.\\
$[2]$ Y. Animov, \textit{Differential Geometry and Topology of Curves}, CRC Press, 2001.\\
$[3]$ J. Foran, \textit{Fundamentals of Real Analysis}, CRC Press, 1991.\\
$[4]$ F.R. Harvey, H.B. Lawson, Jr., \textit{Dirichlet duality and the non-linear Dirichlet problem}, Comm. on Pure and Applied Math. 62 (2009), 396-443.\\
$[5]$ F.R. Harvey, H.B. Lawson, Jr., \textit{Notes on the differentiation of Quasi-Convex Functions}, 2014.\\
$[6]$ J.-B. Hiriart-Urruty, C. Lemar\'{a}chal, \textit{Convex Analysis and Minimization Algorithms}, Vol.I and II, Springer, 1993.\\
$[7]$ R.T. Rockafellar, \textit{Convex Analysis}, Princeton University Press, 1970.\\
$[8] $ C.A. Rogers, \textit{Hausdorff measures},  Cambridge University Press,1970.\\
$[9]$ Z. Slodkowski, \textit{The Bremermann-Dirichlet Problem for $q-$Plurisubharmonic Functions}, Analli della Scuola Normale Superiore di Pisa, Classe di Scienze, $4^e$ s\'{e}rie, tome 11, no. 2 (1984), p. 303-326. \vspace{10mm}\\ \normalsize
\author{University of Maryland \\
\thanks{\texttt{mdellato@math.umd.edu}}}
\affil{Department of Mathematics, University of Maryland, College Park}
\end{document}